\documentclass{amsart}
\usepackage{amssymb}
\usepackage{enumerate}
\usepackage{color}

\newtheorem{thm}{Theorem}[section]

\newtheorem{lem}[thm]{Lemma}

\newtheorem{defi}[thm]{Definition}
\newtheorem{cor}[thm]{Corollary}

\newtheorem{quest}[thm]{Question}

\begin{document}
	
\title{Connes Integration Formula without singular traces}

\author[]{F. Sukochev}
\address{School of Mathematics and Statistics, University of New South Wales, Kensington,  2052, Australia}
\email{f.sukochev@unsw.edu.au}

\author[]{D. Zanin}
\address{School of Mathematics and Statistics, University of New South Wales, Kensington,  2052, Australia}
\email{d.zanin@unsw.edu.au}
	
\begin{abstract}
A version of Connes Integration Formula which provides concrete asymptotics of the eigenvalues is given. This radically extending the class of quantum-integrable functions on compact Riemannian manifolds.
\end{abstract}	

\maketitle

\section{Introduction}

Traditional noncommutative integration theory is based on normal linear functionals on von Neumann algebras, see \cite{Se2} and the monographs \cite{BR}, \cite{Ped}, \cite{T} (among many).  So it is somewhat surprising, and a disparity, that the formula (for some trace $\varphi$ on the ideal $\mathcal{L}_{1,\infty}$ and for some fixed $T\in\mathcal{L}_{1,\infty}$)
\begin{equation}\label{generic integration formula}
\varpi(A)=\varphi(AT),\quad A\in B(H),
\end{equation}
with its obscured normality, and not the formula (for some fixed $T\in\mathcal{L}_1$)
$$\varpi(A)={\rm Tr}(AT),\quad A\in B(H),$$
appears as the analogue of integration in noncommutative geometry.  That it does is due to numerous results of
A.~Connes achieved with the Dixmier trace, see \cite{C3}, \S IV in \cite{bigredbook} and \cite{CM} (as a sample).  In Connes' noncommutative geometry the formula \eqref{generic integration formula} has been termed the noncommutative integral, see e.g. p.297 in \cite{GBVF} or p.478 in \cite{Hawk}, due to the link to noncommutative residues in differential geometry described by the theorem of Connes, see Theorem 1 in \cite{C3} or Theorem 7.18 on p.293 in \cite{GBVF}. Below, we describe the special case of Connes theorem, known as Connes Integration Formula.

Let $(X,g)$ be a compact $d$-dimensional Riemannian manifold (see e.g. p.11 in \cite{Rosenberg}) and let ${\rm vol}_g$ be the Riemannian volume (see e.g. p.15 in \cite{Rosenberg}). Let $\Delta_g$ be the Laplace-Beltrami operator on $X$ (see Section 2.4 in \cite{Taylor}). An integration formula due to Connes\footnote{Connes used a special class of traces known as Dixmier traces. The formula, as stated, appears in Theorem 11.7.10 in \cite{LSZ-book}.} reads as follows.
$$\varphi(M_f(1-\Delta_g)^{-\frac{d}{2}})=\frac{{\rm Vol}(\mathbb{S}^{d-1})}{d(2\pi)^d}\int_X fd{\rm vol}_g,\quad f\in C^{\infty}(X),$$
where ${\rm vol}_g$ is the volume form corresponding to the Riemannian metric $g$ and where $\varphi$ is a positive normalised trace on the ideal $\mathcal{L}_{1,\infty}.$ By the tracial property, one can rewrite the preceding formula as
$$\varphi((1-\Delta_g)^{-\frac{d}{4}}M_f(1-\Delta_g)^{-\frac{d}{4}})=\frac{{\rm Vol}(\mathbb{S}^{d-1})}{d(2\pi)^d}\int_X fd{\rm vol}_g,\quad f\in C^{\infty}(X).$$

The following question was asked by Connes during his 70-th anniversary conference.
\begin{quest}\label{connes question} Is it possible to prove directly an asymptotic formula for the eigenvalues of the operator
$$M_f(1-\Delta_g)^{-\frac{d}{2}}\mbox{ or }(1-\Delta_g)^{-\frac{d}{4}}M_f(1-\Delta_g)^{-\frac{d}{4}}$$
which allows to deduce the Integration Formula without involving ultrafilters (extended limits, Banach limits or similar tools)?
\end{quest}

In this paper, we answer the Question \ref{connes question} in the affirmative. In fact, the following result radically extends the class of functions which admit a non-commutative integral. It allows to write Connes Integration Formula without using singular traces.

Consider the finite measure space $(X,{\rm vol}_g).$ Let $M(t)=t\log(e+t),$ $t>0,$ and let $L_M(X,{\rm vol}_g)$ be the Orlicz space associated with this function.

\begin{thm}\label{cif manifold} Let $(X,g)$ be a compact $d$-dimensional Riemannian manifold and let $\Delta_g$ be the Laplace-Beltrami operator. For every real-valued $f\in L_M(X,{\rm vol}_g),$ we have
$$\lim_{t\to\infty}t\mu\Big(t,\Big((1-\Delta_g)^{-\frac{d}{4}}M_f(1-\Delta_g)^{-\frac{d}{4}}\Big)_+\Big)=\frac{{\rm Vol}(\mathbb{S}^{d-1})}{d(2\pi)^d}\int_X f_+d{\rm vol}_g,$$ 
$$\lim_{t\to\infty}t\mu\Big(t,\Big((1-\Delta_g)^{-\frac{d}{4}}M_f(1-\Delta_g)^{-\frac{d}{4}}\Big)_-\Big)=\frac{{\rm Vol}(\mathbb{S}^{d-1})}{d(2\pi)^d}\int_X f_-d{\rm vol}_g.$$ 	
\end{thm}	
Here, $t\to\mu(t,T)$ is the singular value function of the operator $T$ (see the definition in Subsection \ref{ideal subsection}).

Theorem \ref{cif manifold} strenghtens/complements a number of earlier results in the literature (e.g. Theorems 2.8 and 5.9 in \cite{LPS}, Theorem 1.7 in \cite{KLPS}, Theorem 11.7.10 in \cite{LSZ-book}, Theorems 1.1 and 1.2 in \cite{LSZ-last-kalton}).

We caution the reader that the symmetric form of Connes Integration Formula is necessary if one wants to integrate the functions which do not belong to $L_2(X,{\rm vol}_g).$ For functions from $L_2(X,{\rm vol}_g),$ the (asymmetric) integration formula appeared in Corollary 7.24 in \cite{KLPS} (see also Theorem 11.7.10 in \cite{LSZ-book}). It is established in Theorem 2.5 in \cite{LPS} that, for $f\notin L_2(X,{\rm vol}_g),$ the operator $M_f(1-\Delta_g)^{-\frac{d}{2}}$ cannot belong to $\mathcal{L}_2$ (and, therefore, to $\mathcal{L}_{1,\infty}$). Thus, Theorem \ref{cif manifold} yields a non-trivial extension of integration formulae from \cite{LPS} and \cite{KLPS} to the realm where the latter formulae are false.

The proof of Theorem \ref{cif manifold} is based on a number of recent advances and dicovery of a connection of old Birman and Solomyak results with our theme.

In \cite{Solomyak1995}, Solomyak proved the following specific Cwikel-type estimate in even dimension:
\begin{equation}\label{solomyak estimate}
\Big\|M_f(1-\Delta_{\mathbb{T}^d})^{-\frac{d}{4}}\Big\|_{2,\infty}\leq c_d\|f\|_{L_M^{(2)}},\quad f\in L_M^{(2)}(\mathbb{T}^d).
\end{equation}
Here, $\mathcal{L}_{2,\infty}$ is the weak Hilbert-Schmidt class and $L_M^{(2)}$ is the $2$-convexification of the Orlicz space $L_M.$ This result was extended to an arbitrary dimension in \cite{SZ-solomyak}.

In this paper, we extend the estimate \eqref{solomyak estimate} to an arbitrary compact Riemannian manifold.

\begin{thm}\label{solomyak cwikel estimate manifold} Let $(X,g)$ be a compact $d$-dimensional Riemannian manifold and let $\Delta_g$ be the Laplace-Beltrami operator. For every $f\in L_M^{(2)}(X,{\rm vol}_g),$ we have
$$\Big\|M_f(1-\Delta_g)^{-\frac{d}{4}}\Big\|_{2,\infty}\leq C_{X,g}\|f\|_{L_M^{(2)}}.$$
\end{thm}

%

%
%

In Section \ref{cif manifold section}, we prove Theorem \ref{cif manifold} and compare it with the work of Birman and Solomyak \cite{BS70}.

For Euclidean space $\mathbb{R}^d,$ the following asymptotic formula was established by Birman and Solomyak \cite{BS70}. If $f\in L_p(\mathbb{R}^d),$ $p>2q,$ $q>1,$ is a positive compactly supported function, then Theorem 1 in \cite{BS70} (taken with  $\Phi=1$ and $\alpha=-\frac{d}{q'}$) yields
$$\lim_{n\to\infty}n^{\frac1q}\mu\Big(n,M_{f^{\frac12}}(-\Delta_{\mathbb{R}^d})^{-\frac{d}{2q}}M_{f^{\frac12}}\Big)=c_q\|f\|_q.$$
Should such a formula hold for $q=1,$ it would imply a version of Theorem \ref{cif manifold} for compactly supported functions on Euclidean spaces.

However, we are working on a compact manifold, not on Euclidean space and we do not have the Birman-Solomyak asymptotic formula for $p=1.$ For these reasons, our approach below is very different to that in \cite{BS70}.

\section{Preliminaries}

\subsection{Trace ideals}\label{ideal subsection}
The following material is standard; for more details we refer the reader to \cite{LSZ-book,Simon-book}.
Let $H$ be a complex separable infinite dimensional Hilbert space, and let $B(H)$ denote the set of all bounded operators on $H$, and let $K(H)$ denote the ideal of compact operators on $H.$ Given $T\in K(H),$ the sequence of singular values $\mu(T) = \{\mu(k,T)\}_{k=0}^\infty$ is defined as:
\begin{equation*}
\mu(k,T) = \inf\{\|T-R\|_{\infty}:\quad \mathrm{rank}(R) \leq k\}.
\end{equation*}
It is often convenient to identify the sequence $(\mu(k,T))_{k\geq0}$ with a step function $\sum_{k\geq0}\mu(k,T)\chi_{(k,k+1)}.$

Let $p \in (0,\infty).$ The weak Schatten class $\mathcal{L}_{p,\infty}$ is the set of operators $T$ such that $\mu(T)$ is in the weak $L_p$-space $l_{p,\infty}$, with the quasi-norm:
\begin{equation*}
	\|T\|_{p,\infty} = \sup_{k\geq 0} (k+1)^{\frac1p}\mu(k,T) < \infty.
\end{equation*}
Obviously, $\mathcal{L}_{p,\infty}$ is an ideal in $B(H).$ We also have the following form
of H\"older's inequality,
\begin{equation}\label{weak holder}
\|TS\|_{r,\infty} \leq c_{p,q}\|T\|_{p,\infty}\|S\|_{q,\infty}
\end{equation}
where $\frac{1}{r}=\frac{1}{p}+\frac{1}{q}$, for some constant $c_{p,q}$. Indeed, this follows from the definition of these quasi-norms and the inequality (see e.g. \cite[Proposition 1.6]{Fack1982}, \cite[Corollary 2.2]{GohbergKrein})
$$\mu(2n,TS)\leq \mu(n,T)\mu(n,S),\quad n\geq 0.$$

The closure of the set of all finite rank operators in $\mathcal{L}_{p,\infty}$ is called the separable part of $\mathcal{L}_{p,\infty}$ and is denoted by $(\mathcal{L}_{p,\infty})_0.$

The ideal of particular interest is $\mathcal{L}_{1,\infty}$, and we are concerned with traces on this ideal. For more details, see \cite[Section 5.7]{LSZ-book} and \cite{SSUZ2015}. A linear functional $\varphi:\mathcal{L}_{1,\infty}\to\mathbb{C}$ is called a trace if it is unitarily invariant. That is, for all unitary operators $U$ and for all $T\in\mathcal{L}_{1,\infty}$ we have that $\varphi(U^{\ast}TU) = \varphi(T)$. It follows that for all bounded operators $B$ we have $\varphi(BT)=\varphi(TB).$  

Every trace $\varphi:\mathcal{L}_{1,\infty}\to\mathbb{C}$ vanishes on the ideal of finite rank operators (such traces are called singular). In fact, $\varphi$ vanishes on the ideal $\mathcal{L}_1$ (see \cite{DFWW} or \cite{LSZ-book}). For the state of the art in the theory of singular traces and their applications in Non-commutative Geometry, we refer the reader to the survey \cite{LSZ-survey}.

\subsection{Sobolev spaces on compact manifolds}\label{sobolev manifold subsection}

In our definition of Sobolev space on compact manifolds, we follow \cite{Taylor}. 

Let $(X,g)$ be a compact $d$-dimensional Riemannian manifold and let ${\rm vol}_g$ be the Riemannian volume. If $u\in L_2(X,{\rm vol}_g),$ we say $u\in W^{s,2}(X,{\rm vol}_g)$ provided that, on each chart $U\subset X,$ every $\phi\in C_c^{\infty}(U),$ the element $\phi u$ belongs to $W^{s,2}(U)$ (if $U$ is identified with its image in $\mathbb{R}^d$). By the invariance under coordinate changes derived in Section 4.2 in \cite{Taylor}, it suffices to work with any single coordinate cover of $X.$ If $s=m,$ a nonnegative integer, then $W^{m,2}(X,{\rm vol}_g)$ is equal to the set of all $u\in L_2(X,{\rm vol}_g)$ such that, for any smooth vector fields $(X_l)_{l=1}^m,$ we have $X_1\cdots X_lu\in L_2(X,{\rm vol}_g).$

The following theorem is stated in Section 4.3 in \cite{Taylor}.
For the details of compex interpolation, we refer to the book \cite{KPS}.

\begin{thm}\label{sobolev manifold interpolation thm} For every $m\in\mathbb{Z}_+,$ for every $0<\theta<1,$ we have
	$$[L_2(X,{\rm vol}_g),W^{m,2}(X,{\rm vol}_g)]_{\theta}=W^{m\theta,2}(X,{\rm vol}_g).$$ 
\end{thm}

\section{Proof of Theorem \ref{solomyak cwikel estimate manifold}}\label{manifold section}

Having Theorem 1.1 in \cite{SZ-solomyak} at hands, we derive Theorem \ref{solomyak cwikel estimate manifold}. This part is based on and simultaneously improves upon the similar argument in \cite{LSZ-last-kalton}.

We begin with the simplest case of Theorem \ref{solomyak cwikel estimate manifold} when the manifold $X$ is given by $\mathbb{T}^d$ with the flat metric.

\begin{lem}\label{solomyak restated lemma} If $f\in L_{\infty}(\mathbb{T}^d),$ then
$$\|M_f\|_{\mathcal{L}_{2,\infty}(W^{\frac{d}{2},2}(\mathbb{T}^d)\to L_2(\mathbb{T}^d))}\leq c_d\|f\|_{L_M^{(2)}}.$$
\end{lem}
\begin{proof} We view the operator
$$(M_f)_{W^{\frac{d}{2},2}(\mathbb{T}^d)\to L_2(\mathbb{T}^d)}$$
as a combination
$$\Big(M_f(1-\Delta)^{-\frac{d}{4}}\Big)_{L_2(\mathbb{T}^d)\to L_2(\mathbb{T}^d)}\circ (1-\Delta)^{\frac{d}{4}}_{W^{\frac{d}{2},2}(\mathbb{T}^d)\to L_2(\mathbb{T}^d)}.$$
By definition, the operator
$$(1-\Delta)^{\frac{d}{4}}_{W^{\frac{d}{2},2}(\mathbb{T}^d)\to L_2(\mathbb{T}^d)}$$
is an isometry between Hilbert spaces, and therefore
$$\|M_f\|_{\mathcal{L}_{2,\infty}(W^{\frac{d}{2},2}(\mathbb{T}^d)\to L_2(\mathbb{T}^d))}=\Big\|M_f(1-\Delta)^{-\frac{d}{4}}\Big\|_{\mathcal{L}_{2,\infty}(L_2(\mathbb{T}^d)\to L_2(\mathbb{T}^d))}.$$
Obviously,
$$\Big\|M_f(1-\Delta)^{-\frac{d}{4}}\Big\|_{2,\infty}^2=\Big\|\Big|M_f(1-\Delta)^{-\frac{d}{4}}\Big|^2\Big\|_{1,\infty}=\Big\|(1-\Delta)^{-\frac{d}{4}}M_{|f|^2}(1-\Delta)^{-\frac{d}{4}}\Big\|_{1,\infty}.$$
Appealing now to Theorem 1.1 in \cite{SZ-solomyak}, we obtain
$$\Big\|M_f(1-\Delta)^{-\frac{d}{4}}\Big\|_{2,\infty}^2\leq c_d\big\||f|^2\big\|_{L_M}=c_d\|f\|_{L_M^{(2)}}^2.$$
Combining three preceding displays, we complete the proof.
\end{proof}

Since $X$ is a compact manifold, we may assume without loss of generality that our atlas consists of finitely many charts. Furthermore, we assume that, for each chart $(U,\gamma)$ in our atlas, the set $\gamma(U)$ is bounded. Thus, $\gamma(U)$ is compactly supported in a sufficiently large open box $(-N,N)^d.$ By applying a dilation if necessary, we may assume without loss of generality that $\gamma(U)$ is compactly supported in $(-\pi,\pi)^d.$ By identifying the edges of $(-\pi,\pi)^d,$ we may view $\gamma$ as a continuous function $\gamma:U\to\mathbb{T}^d.$

Further, we assume that, in every chart $(U,\gamma)$ from our atlas, the metric tensor is bounded from above and from below. Hence, the measure ${\rm vol}_g\circ\gamma^{-1}$ is equivalent to the Haar measure in the following sense: those measures are mutually absolutely continuous and Radon-Nikodym derivatives are bounded. These considerations yield the following two lemmas.

\begin{lem}\label{l2 gamma lemma} Let $(U,\gamma)$ be a chart.
	\begin{enumerate}[{\rm (i)}]
		\item A linear mapping $U_{\gamma}:L_2(X,{\rm vol}_g)\to L_2(\mathbb{T}^d)$ defined by the formula
		$$U_{\gamma}\xi=\chi_{\gamma(U)}\cdot (\xi\circ\gamma^{-1}),\quad \xi\in L_2(X,g),$$
		is bounded;
		\item A linear mapping $V_{\gamma}:L_2(\mathbb{T}^d)\to L_2(X,{\rm vol}_g)$ defined by the formula
		$$V_{\gamma}\xi=\chi_U\cdot(\xi\circ\gamma),\quad \xi\in L_2(\mathbb{T}^d),$$
		is bounded;
	\end{enumerate}	
\end{lem}

\begin{lem}\label{e gamma lemma} Let $(U,\gamma)$ be a chart and and let $U_{\gamma}$ be as in Lemma \ref{l2 gamma lemma}. For every $f\in L_{\infty}(X,g),$  we have
	$$\|U_{\gamma}f\|_{L_M^{(2)}(\mathbb{T}^d)}\leq C_{U,\gamma}\|f\|_{L_M^{(2)}(X,{\rm vol}_g)}.$$ 
\end{lem}

Our next lemma is a variant of Lemma \ref{l2 gamma lemma} for Sobolev spaces.

\begin{lem}\label{sobolev gamma lemma} Let $(U,\gamma)$ be a chart and let $U_{\gamma}$ be as in Lemma \ref{l2 gamma lemma}. If $\phi\in C^{\infty}_c(U),$ then
	$$U_{\gamma}M_{\phi}:W^{\frac{d}{2},2}(X,{\rm vol}_g)\to W^{\frac{d}{2},2}(\mathbb{T}^d)$$
	is everywhere defined bounded mapping.	
\end{lem}
\begin{proof} This is, essentially, a definition of Sobolev space $W^{\frac{d}{2},2}(X,{\rm vol}_g)$ (see Subsection \ref{sobolev manifold subsection} above).
\end{proof}

\begin{lem}\label{local sobolev lemma} Let $(U,\gamma)$ be a chart and let $K\subset U$ be compact. For every $f\in L_{\infty}(X,{\rm vol}_g)$ supported in $K,$ we have
	$$\Big\|M_f\Big\|_{\mathcal{L}_{2,\infty}(W^{\frac{d}{2},2}(X,{\rm vol}_g)\to L_2(X,{\rm vol}_g))}\leq C_{K,U,\gamma,X,g}\|f\|_{L_M^{(2)}(X,{\rm vol}_g)}.$$
\end{lem}
\begin{proof} Let $U_{\gamma}$ and $V_{\gamma}$ be as in Lemma \ref{l2 gamma lemma}. Note that
	$$V_{\gamma}U_{\gamma}=M_{\chi_U}.$$
	
	Choose $\phi\in C^{\infty}_c(U)$ such that $\phi=1$ on $K.$ We write
	$$M_f=M_fM_{\phi},\quad M_f=M_fV_{\gamma}U_{\gamma},\quad M_f=V_{\gamma}U_{\gamma}M_f.$$	
	Thus,
	$$M_f=V_{\gamma}U_{\gamma}M_fV_{\gamma}U_{\gamma}M_{\phi}=V_{\gamma}\cdot U_{\gamma}M_fV_{\gamma}\cdot U_{\gamma}M_{\phi}=V_{\gamma}\cdot M_{U_{\gamma}f}\cdot U_{\gamma}M_{\phi}.$$
	Hence,
	$$
	\Big\|M_f\Big\|_{\mathcal{L}_{2,\infty}(W^{\frac{d}{2},2}(X,{\rm vol}_g)\to L_2(X,{\rm vol}_g))}\leq ABC,
	$$
	where
	$$
	A=\|V_{\gamma}\|_{L_2(\mathbb{T}^d)\to L_2(X,{\rm vol}_g)},\ B=\|M_{U_{\gamma}f}\|_{\mathcal{L}_{2,\infty}(W^{\frac{d}{2},2}(\mathbb{T}^d)\to L_2(\mathbb{T}^d))},\ {\rm and}
	$$
	$$ C=  \|U_{\gamma}M_{\phi}\|_{W^{\frac{d}{2},2}(X,{\rm vol}_g)\to W^{\frac{d}{2},2}(\mathbb{T}^d)}.
	$$
	The first factor $A$ is finite by Lemma \ref{l2 gamma lemma} (it depends on $X,$ $U$ and $\gamma$). The third factor $C$ is finite by Lemma \ref{sobolev gamma lemma} (it depends not only on $X,$ $U$ and $\gamma,$ but also on $\phi$ and, hence, on $K$). It follows from Lemma \ref{solomyak restated lemma} and Lemma \ref{e gamma lemma} that
	$$B=\|M_{U_{\gamma}f}\|_{\mathcal{L}_{2,\infty}(W^{\frac{d}{2},2}(\mathbb{T}^d)\to L_2(\mathbb{T}^d))}\leq c_d\|U_{\gamma}f\|_{L_M^{(2)}(\mathbb{T}^d)}\leq c_dC_{U,\gamma}\|f\|_{L_M^{(2)}(X,{\rm vol}_g)}.$$
	Combining these estimates, we complete the proof.
\end{proof}

The next lemma extends the result of Lemma \ref{local sobolev lemma} by removing the assumption that $f$ is supported in a chart.

\begin{lem}\label{global sobolev lemma} If $f\in L_{\infty}(X,g),$ then 
	$$\|M_f\|_{\mathcal{L}_{2,\infty}(W^{\frac{d}{2},2}(X,{\rm vol}_g)\to L_2(X,{\rm vol}_g))}\leq C_{X,g}\|f\|_{L_M^{(2)}(X,{\rm vol}_g)}.$$
\end{lem}
\begin{proof} Let $\{\phi_k\}_{k=1}^N$ be a partition of unity subordinate to an atlas of bounded charts $(U_k,\gamma_k).$ We write
	$$M_f=\sum_{i=k}^N M_{f\phi_k}.$$
	Recall the triangle inequality in $\mathcal{L}_{2,\infty}$ (it can be found e.g. in \cite{LSZ-book}):
	$$\|\sum_{k=1}^nA_k\|_{2,\infty}\leq 2\sum_{k=1}^n\|A_k\|_{2,\infty}.$$
	Thus,
	$$\|M_f\|_{\mathcal{L}_{2,\infty}(W^{\frac{d}{2},2}(X,{\rm vol}_g)\to L_2(X,{\rm vol}_g))}\leq 2\sum_{k=1}^N\|M_{f\phi_k}\|_{\mathcal{L}_{2,\infty}(W^{\frac{d}{2},2}(X,{\rm vol}_g)\to L_2(X,{\rm vol}_g))}.$$
	By Lemma \ref{local sobolev lemma}, we have
	$$\|M_{f\phi_k}\|_{\mathcal{L}_{2,\infty}(W^{\frac{d}{2},2}(X,{\rm vol}_g)\to L_2(X,{\rm vol}_g))}\leq C_{{\rm supp}(\phi_k),U_k,\gamma_k,X,g}\|f\phi_k\|_{L_M^{(2)}(X,{\rm vol}_g)}.$$
	Thus,
	\begin{align*}
	\|M_f\|&_{\mathcal{L}_{2,\infty}(W^{\frac{d}{2},2}(X,{\rm vol}_g)\to L_2(X,{\rm vol}_g))}\leq 2\sum_{k=1}^NC_{{\rm supp}(\phi_k),U_k,\gamma_k,X,g}\|f\phi_k\|_{L_M^{(2)}(X,{\rm vol}_g)}\\& \leq
	2\Big(\sum_{k=1}^NC_{E,{\rm supp}(\phi_k),U_k,\gamma_k,X,g}\|\phi_k\|_{\infty}\Big)\|f\|_{L_M^{(2)}(X,{\rm vol}_g)}.
	\end{align*}
	
\end{proof}

Not every author defines Sobolev space on compact manifolds as in \cite{Taylor}. An equally important definition is via powers of Laplace-Beltrami operator $\Delta_g.$ Those definitions are known to be equivalent. We only need one side of this equivalence as established in the following lemma.

\begin{lem}\label{laplace sobolev lemma} For every $s>0,$ the operator $(1-\Delta_g)^{-\frac{s}{2}}$ is a well defined and bounded mapping from $L_2(X,{\rm vol}_g)$ to $W^{s,2}(X,{\rm vol}_g).$
\end{lem}
\begin{proof} Let $H$ be a Hilbert space and let $P:{\rm dom}(P)\to H$ be self-adjoint operator. Suppose, in addition, that $P\geq1.$ The space $H_s={\rm dom}(P^s)$ becomes a Hilbert space when equipped with the norm $\xi\to\|P^s\xi\|,$ $\xi\in H^s.$ It is immediate from this definition that $P^{-s}:H\to H_s.$ Note the complex interpolation\footnote{By the spectral theorem, it suffices to check it when $H=L_2(\Omega,\nu)$ and when $P$ is a multiplication operator. In this case, it is a standard result about complex interpolation of weighted $L_2$-spaces (see Theorem 5.4.1 in \cite{BerghLoefstrom}).}
	$$[H_{s_1},H_{s_2}]_{\theta}=H_{(1-\theta)s_1+\theta s_2},\quad s_1\neq s_2\in\mathbb{R}_+.$$
	In particular,
	$$[H_0,H_m]_{\theta}=H_{\theta m},\quad m\in\mathbb{Z}_+.$$
	
	Now, let $H=L_2(X,{\rm vol}_g)$ and $P=(1-\Delta_g)^{\frac12}.$ Obviously, $H_{2m}={\rm dom}((1-\Delta_g)^m)$ for $m\in\mathbb{Z}_+.$ By elliptic regularity, ${\rm dom}((1-\Delta_g)^m)\subset W^{2m,2}(X,{\rm vol}_g)$ (see e.g. Theorem 19.5.1 in \cite{Hor3}). By G\aa rding inequality (see e.g. Theorem 2.44 in \cite{Rosenberg}), we have
	$$\|(1-\Delta_g)^mu\|_{L_2(X,{\rm vol}_g)}\approx_{m,g}\|u\|_{W^{2m,2}(X,{\rm vol}_g)},\quad u\in W^{2m,2}(X,{\rm vol}_g).$$
	Thus, $H_{2m}=W^{2m,2}(X,{\rm vol}_g).$ By Theorem \ref{sobolev manifold interpolation thm}, we have
	$$[L_2(X,{\rm vol}_g),W^{2m,2}(X,{\rm vol}_g)]_{\theta}=W^{2\theta m,2}(X,{\rm vol}_g),\quad m\in\mathbb{Z}_+.$$
	Take $s\in\mathbb{R}_+,$ choose integer $2m>s$ and let $\theta=\frac{s}{2m}.$ We have
	$$H_s=W^{s,2}(X,{\rm vol}_g),\quad s\in\mathbb{R}_+.$$
	It follows immediately that
	$$(1-\Delta_g)^{-\frac{s}{2}}=P^{-s}:L_2(X,{\rm vol}_g)=H\to H_s=W^{s,2}(X,{\rm vol}_g),\quad s\in\mathbb{R}_+.$$
\end{proof}

\begin{proof}[Proof of Theorem \ref{solomyak cwikel estimate manifold}] {\bf Step 1:} Suppose first that $f\in L_{\infty}(X,g).$ Obviously,
	\begin{align*}
	\Big(M_f(1-\Delta_g)^{-\frac{d}{4}}\Big)_{L_2(X,{\rm vol}_g)\to L_2(X,{\rm vol}_g)}&=(M_f)_{W^{\frac{d}{2},2}(X,{\rm vol}_g)\to L_2(X,{\rm vol}_g)}\\ &\circ \Big((1-\Delta_g)^{-\frac{d}{4}}\Big)_{L_2(X,{\rm vol}_g)\to W^{\frac{d}{2},2}(X,{\rm vol}_g)}.
	\end{align*}
	By Lemma \ref{laplace sobolev lemma}, 
	$$\Big((1-\Delta_g)^{-\frac{d}{4}}\Big)_{L_2(X,{\rm vol}_g)\to W^{\frac{d}{2},2}(X,{\rm vol}_g)}$$
	is bounded. It follows that
	$$\Big\|M_f(1-\Delta_g)^{-\frac{d}{4}}\Big\|_{\mathcal{L}_{2,\infty}(L_2(X,{\rm vol}_g))}\leq AB,\ {\rm where}
	$$
	$$A=\Big\|M_f\Big\|_{\mathcal{L}_{2,\infty}(W^{\frac{d}{2},2}(X,{\rm vol}_g)\to L_2(X,{\rm vol}_g))},\ B=\Big\|(1-\Delta_g)^{-\frac{d}{4}}\Big\|_{L_2(X,{\rm vol}_g)\to W^{\frac{d}{2},2}(X,{\rm vol}_g)}.$$
	The assertion follows now from Lemma \ref{global sobolev lemma}.
	
	{\bf Step 2:} Suppose now that $f\in L_M^{(2)}(X,{\rm vol}_g).$ We claim that the operator
	$$M_f(1-\Delta_g)^{-\frac{d}{4}}$$
	is everywhere defined and bounded on $L_2(X,{\rm vol}_g).$ Moreover, we claim that
	$$\Big\|M_f(1-\Delta_g)^{-\frac{d}{4}}\Big\|_{\infty}\leq C_{X,g}\|f\|_{L_M^{(2)}(X,{\rm vol}_g)}.$$
	
	For every $\xi\in L_2(X,{\rm vol}_g),$ the function
	$$f\cdot (1-\Delta_g)^{-\frac{d}{4}}\xi$$
	is well-defined and measurable (as a product of two measurable functions). We show that this function belongs to $L_2(X,{\rm vol}_g)$ (and also establish the estimate for its $L_2$-norm from the above). Set
	$$f_n=\min\{|f|,n\},\quad n\in\mathbb{N}.$$
	If follows from the Fatou Theorem that
	$$\|f\cdot (1-\Delta_g)^{-\frac{d}{4}}\xi\|_{L_2(X,{\rm vol}_g)}=\sup_{n\in\mathbb{N}}\|f_n\cdot (1-\Delta_g)^{-\frac{d}{4}}\xi\|_{L_2(X,{\rm vol}_g)}.$$
	On the other hand, it follows from Step 1 that
	$$\|f_n\cdot (1-\Delta_g)^{-\frac{d}{4}}\xi\|_{L_2(X,{\rm vol}_g)}\leq \|M_{f_n}(1-\Delta_g)^{-\frac{d}{4}}\|_{\infty}\|\xi\|_{L_2(X,{\rm vol}_g)}\leq$$
	$$\leq\|M_{f_n}(1-\Delta_g)^{-\frac{d}{4}}\|_{2,\infty}\|\xi\|_{L_2(X,{\rm vol}_g)}\leq$$
	$$\leq C_{X,g}\|f_n\|_{L_M^{(2)}(X,{\rm vol}_g)}\|\xi\|_{L_2(X,{\rm vol}_g)}\leq C_{X,g}\|f\|_{L_M^{(2)}(X,{\rm vol}_g)}\|\xi\|_{L_2(X,{\rm vol}_g)}.$$
	It follows that
	$$\|f\cdot (1-\Delta_g)^{-\frac{d}{4}}\xi\|_{L_2(X,{\rm vol}_g)}\leq C_{X,g}\|f\|_{L_M^{(2)}(X,{\rm vol}_g)}\|\xi\|_{L_2(X,{\rm vol}_g)}.$$
	This proves the claim.
	
	{\bf Step 3:} It follows from Step 1 that
	$$\Big\|M_{f_n}(1-\Delta_g)^{-\frac{d}{4}}\Big\|_{\mathcal{L}_{2,\infty}(L_2(X,{\rm vol}_g))}\leq C_{X,g}\|f_n\|_{L_M^{(2)}(X,{\rm vol}_g)}\leq C_{X,g}\|f\|_{L_M^{(2)}(X,{\rm vol}_g)}.$$
	Using Step 2, we obtain
	$$\Big\|M_{f_n}(1-\Delta_g)^{-\frac{d}{4}}-M_{|f|}(1-\Delta_g)^{-\frac{d}{4}}\Big\|_{\infty}\leq C_{E,X,g}\big\|f_n-|f|\big\|_{L_M^{(2)}(X,{\rm vol}_g)}.$$
	Thus,
	$$M_{f_n}(1-\Delta_g)^{-\frac{d}{4}}\to M_{|f|}(1-\Delta_g)^{-\frac{d}{4}}$$
	in the uniform norm as $n\to\infty.$ By the Fatou property (see \cite{Simon-book} and \cite{LSZ-book}) for the space $\mathcal{L}_{2,\infty},$ it follows that
	$$\Big\|M_{|f|}(1-\Delta_g)^{-\frac{d}{4}}\Big\|_{\mathcal{L}_{2,\infty}(L_2(X,{\rm vol}_g))}\leq C_{X,g}\|f\|_{L_M^{(2)}(X,{\rm vol}_g)}.$$
	This completes the proof.
\end{proof}

\section{Abstract lemmas on the asymptotics of eigenvalues}

The starting point is the assertion known (in various forms) since at least 1940's.

\begin{lem}\label{limit of good operators mu lemma} Let $(T_n)_{n\geq0}\subset\mathcal{L}_{1,\infty}$ be  such that
$$\lim_{t\to\infty}t\mu(t,T_n)=\alpha_n,\quad n\geq0.$$
If $T_n\to T$ in $\mathcal{L}_{1,\infty},$ then $\alpha_n\to\alpha$ and
$$\lim_{t\to\infty}t\mu(t,T)=\alpha.$$
\end{lem}
\begin{proof} To lighten the notations, we assume without loss of generality that $\|T_n\|_{1,\infty}\leq 1$ for $n\geq0.$ This also means that $\|T\|_{1,\infty}\leq 1.$
	
{\bf Step 1:} We show that $(\alpha_n)_{n\geq0}$ converges (its limit will be denoted by $\alpha$).
	
Fix $\epsilon\in(0,1)$ and choose $N$ such that $\|T_n-T_m\|_{1,\infty}\leq\epsilon^2$ for $n,m\geq N.$ We have
$$\mu(t,T_n)=\mu(t,T_m+(T_n-T_m))\leq \mu(\frac{t}{1+\epsilon},T_m)+\mu(\frac{t\epsilon}{1+\epsilon},T_n-T_m)\leq$$
$$\leq\mu(\frac{t}{1+\epsilon},T_m)+\frac{1+\epsilon}{t\epsilon}\cdot\|T_m-T_n\|_{1,\infty}\leq \leq\mu(\frac{t}{1+\epsilon},T_m)+\frac{2\epsilon}{t}.$$
Thus,
$$\lim_{t\to\infty}t\mu(t,T_n)\leq\lim_{t\to\infty}t\mu(\frac{t}{1+\epsilon},T_m)+2\epsilon.$$
In other words,
$$\alpha_n\leq (1+\epsilon)\alpha_m+2\epsilon.$$
Consequently,
$$\alpha_n-\alpha_m\leq 2\epsilon+\epsilon\alpha_m\leq\epsilon+\epsilon\|T_m\|_{1,\infty}\leq 3\epsilon.$$
Similarly,
$$\alpha_m-\alpha_n\leq 3\epsilon.$$
Finally,
$$|\alpha_m-\alpha_n|\leq 3\epsilon,\quad m,n\geq N.$$
Thus, $(\alpha_n)_{n\geq0}$ is a Cauchy sequence and the claim in Step 1 follows.
	
{\bf Step 2:} We show that
$$\limsup_{t\to\infty}t\mu(t,T)\leq \alpha.$$
	
Fix $\epsilon\in(0,1)$ and choose $N$ such that $\|T_n-T\|_{1,\infty}\leq\epsilon^2$ for $n\geq N.$ We have
$$\mu(t,T)=\mu(t,T+(T-T_n))\leq \mu(\frac{t}{1+\epsilon},T_n)+\mu(\frac{t\epsilon}{1+\epsilon},T-T_n)\leq$$
$$\leq\mu(\frac{t}{1+\epsilon},T_n)+\frac{1+\epsilon}{\epsilon t}\cdot\|T_n-T\|_{1,\infty}\leq \mu(\frac{t}{1+\epsilon},T_n)+2\epsilon t^{-1}.$$
Thus,
$$\limsup_{t\to\infty}t\mu(t,T)\leq\lim_{t\to\infty}t\mu(\frac{t}{1+\epsilon},T_n)+2\epsilon=(1+\epsilon)\alpha_n+2\epsilon\leq \alpha_n+3\epsilon.$$
Passing $n\to\infty,$ we obtain
$$\limsup_{t\to\infty}t\mu(t,T)\leq \alpha+3\epsilon.$$
Since $\epsilon$ is arbitrarily small, the claim in Step 2 follows.

{\bf Step 3:} We show that
$$\alpha\leq\liminf_{t\to\infty}t\mu(t,T).$$
	
Fix $\epsilon\in(0,1)$ and choose $N$ such that $\|T_n-T\|_{1,\infty}\leq\epsilon^2$ for $n\geq N.$ We have
$$\mu(t,T_n)=\mu(t,T_n+(T_n-T))\leq \mu(\frac{t}{1+\epsilon},T)+\mu(\frac{t\epsilon}{1+\epsilon},T_n-T)\leq$$
$$\leq\mu(\frac{t}{1+\epsilon},T)+\frac{1+\epsilon}{\epsilon t}\cdot\|T_n-T\|_{1,\infty}\leq \mu(\frac{t}{1+\epsilon},T)+2\epsilon t^{-1}.$$
Thus,
$$\alpha_n=\lim_{t\to\infty}t\mu(t,T_n)\leq\liminf_{t\to\infty}t\mu(\frac{t}{1+\epsilon},T)+2\epsilon=$$
$$=(1+\epsilon)\liminf_{t\to\infty}t\mu(t,T)+2\epsilon\leq \liminf_{t\to\infty}t\mu(t,T)+3\epsilon.$$
Passing $n\to\infty,$ we obtain
$$\alpha\leq\liminf_{t\to\infty}t\mu(t,T)+3\epsilon.$$
Since $\epsilon$ is arbitrarily small, the claim in Step 3 follows.
	
{\bf Step 4:} Combining the results of Steps 2 and 3, we write
$$\limsup_{t\to\infty}t\mu(t,T)\leq \alpha\leq\liminf_{t\to\infty}t\mu(t,T).$$
Thus,
$$\limsup_{t\to\infty}t\mu(t,T)=\alpha=\liminf_{t\to\infty}t\mu(t,T).$$
The assertion follows immediately.
\end{proof}
The following lemma is a by-product of recent studies of the authors jointly with J.Huang of operator $\theta$-Holder functions for quasi-norms \cite{HSZ}.

\begin{lem}\label{holder lemma} If $T,S\in\mathcal{L}_{1,\infty}$ are self-adjoint operators, then
$$\|T_+-S_+\|_{1,\infty}\leq c_{{\rm abs}}\|T-S\|_{1,\infty}^{\frac12}(\|T\|_{1,\infty}+\|S\|_{1,\infty})^{\frac12}.$$
\end{lem}
\begin{proof} Equation (7) in \cite{HSZ} taken with $f(t)=t_+^{\frac12},$ $t\in\mathbb{R},$ and with $p=\theta=\frac12$ reads as
$$|T_+^{\frac12}-S_+^{\frac12}|^{\frac12}\prec\prec c_{{\rm abs}}|T-S|^{\frac14}.$$
In particular,
$$\|T_+^{\frac12}-S_+^{\frac12}\|_{2,\infty}\leq c_{{\rm abs}}\|T-S\|_{1,\infty}^{\frac12}.$$
It is clear that
$$T_+-S_+=T_+^{\frac12}(T_+^{\frac12}-S_+^{\frac12})+(T_+^{\frac12}-S_+^{\frac12})S_+^{\frac12}.$$
The assertion follows now from H\"older inequality.
\end{proof}

\begin{cor}\label{limit of good operators corollary} Let $(T_n)_{n\geq0}\subset\mathcal{L}_{1,\infty}$ be self-adjoint operators such that
$$\lim_{k\to\infty}k\mu(k,(T_n)_+)=\alpha_n,\quad \lim_{k\to\infty}k\mu(k,(T_n)_-)=\beta_n,\quad n\geq0.$$
If $T_n\to T$ in $\mathcal{L}_{1,\infty},$ then $\alpha_n\to\alpha,$ $\beta_n\to\beta$ and
$$\lim_{k\to\infty}k\mu(k,T_+)=\alpha,\quad \lim_{k\to\infty}k\mu(k,T_-)=\beta.$$
\end{cor}
\begin{proof} By Lemma \ref{holder lemma}, we have $(T_n)_+\to T_+$ and $(T_n)_-\to T_-$ as $n\to\infty.$ The assertion follows from Lemma \ref{limit of good operators mu lemma}.
\end{proof}

\begin{lem}\label{second holder lemma} Let $T,S\in\mathcal{L}_{1,\infty}$ be self-adjoint elements such that $T-S\in(\mathcal{L}_{1,\infty})_0.$ We have that $T_+-S_+\in(\mathcal{L}_{1,\infty})_0.$
\end{lem}
\begin{proof} Equation (7) in \cite{HSZ} taken with $f(t)=t_+^{\frac12},$ $t\in\mathbb{R},$ and with $p=\theta=\frac12$ reads as
$$|T_+^{\frac12}-S_+^{\frac12}|^{\frac12}\prec\prec c_{{\rm abs}}|T-S|^{\frac14}.$$
That is,
$$(n+1)\mu^{\frac12}(n,T_+^{\frac12}-S_+^{\frac12})\leq \sum_{k=0}^n\mu^{\frac12}(k,T_+^{\frac12}-S_+^{\frac12})\leq c_{{\rm abs}}\sum_{k=0}^n\mu^{\frac14}(k,T-S)=o(n^{\frac34})$$
as $n\to\infty.$ In other words,
$$T_+^{\frac12}-S_+^{\frac12}\in (\mathcal{L}_{2,\infty})_0.$$
It is clear that
$$T_+-S_+=T_+^{\frac12}(T_+^{\frac12}-S_+^{\frac12})+(T_+^{\frac12}-S_+^{\frac12})S_+^{\frac12}.$$
The assertion follows now from H\"older inequality.	
\end{proof}

\begin{lem}\label{first good operator fact} Let $T,S\in\mathcal{L}_{1,\infty}$ be self-adjoint elements such that $T-S\in(\mathcal{L}_{1,\infty})_0.$ If
$$\lim_{t\to\infty}t\mu(t,T)=\alpha,{\mbox then }\lim_{t\to\infty}t\mu(t,S)=\alpha.$$
\end{lem}
\begin{proof} Fix $\epsilon>0.$ We have	
$$\mu(t,S)\leq \mu(\frac{t}{1+\epsilon},T)+\mu(\frac{\epsilon t}{1+\epsilon},T-S)\leq \frac{(1+\epsilon)\alpha}{t}+o(t^{-1}),\quad t\to\infty.$$
Thus,
$$\limsup_{t\to\infty}t\mu(t,S)\leq (1+\epsilon)\alpha.$$
On the other hand, we have
$$\mu(t,T)\leq \mu(\frac{t}{1+\epsilon},S)+\mu(\frac{\epsilon t}{1+\epsilon},T-S)= \mu(\frac{t}{1+\epsilon},S)+o(t^{-1}),\quad t\to\infty.$$
Thus,
$$\alpha\leq\liminf_{t\to\infty}t\mu(\frac{t}{1+\epsilon},S)=(1+\epsilon)\liminf_{t\to\infty}t\mu(t,S).$$

Since $\epsilon>0$ is arbitrary, it follows that
$$\limsup_{t\to\infty}t\mu(t,S)\leq \alpha\leq \liminf_{t\to\infty}t\mu(t,S).$$
In other words,
$$\limsup_{t\to\infty}t\mu(t,S)=\alpha=\liminf_{t\to\infty}t\mu(t,S).$$
This completes the proof.
\end{proof}

Below, tensor product of sequences $\alpha$ and $\beta$ is a double sequence given by the formula
$$(\alpha\otimes\beta)(k,l)=\alpha(k)\beta(l),\quad k,l\in\mathbb{Z}_+.$$
Tensor product of a sequence and a function (on $(0,\infty)$) is defined similarly
$$(\alpha\otimes f)(k,s)=\alpha(k)f(s),\quad k\in\mathbb{Z}_+,\quad s\in(0,\infty).$$

\begin{lem}\label{tensor lemma} Let $z(n)=\frac1{n+1}.$ For every finite sequence $\alpha,$ there exists a limit
$$\lim_{t\to\infty}t\mu(t,z\otimes\alpha)=\|\alpha\|_1.$$
\end{lem}
\begin{proof} Let $Z(t)=t^{-1},$ $t>0.$ Note the key fact
$$\mu(Z\otimes\alpha)=\|\alpha\|_1Z.$$

Clearly, $z\leq Z.$ Thus,
\begin{equation}\label{tensor eq0}
t\mu(t,z\otimes\alpha)\leq t\mu(t,Z\otimes\alpha)=\|\alpha\|_1.
\end{equation}

Suppose $\alpha$ is a sequence of length $n.$ We have
$$\mu(t,Z\otimes\alpha)\leq\mu(n,Z\chi_{(0,1)}\otimes\alpha)+\mu(t-n,Z\chi_{(1,\infty)}\otimes\alpha),\quad t>n.$$
Obviously, $Z\chi_{(0,1)}\otimes\alpha$ is supported on a set of measure $n.$ Thus,
$$\mu(n,Z\chi_{(0,1)}\otimes\alpha)=0.$$
Obviously, $\mu(Z\chi_{(1,\infty)})\leq z.$ Thus,
$$\mu(t-n,Z\chi_{(1,\infty)}\otimes\alpha)\leq\mu(t-n,z\otimes\alpha).$$
Consequently,
\begin{equation}\label{tensor eq1}
\|\alpha\|_1=t\mu(t,Z\otimes\alpha)\leq t\mu(t-n,z\otimes\alpha),\quad t>n.
\end{equation}

The assertion follows by combining \eqref{tensor eq0} and \eqref{tensor eq1}.
\end{proof}

In the following lemma (and further below), the notation $\oplus_{k\in\mathbb{Z}}T_k$ is a shorthand for an element $\sum_{k\in\mathbb{Z}}T_k\otimes e_k$ in the von Neumann algebra $B(H)\bar{\otimes}l_{\infty}(\mathbb{Z}).$ Here, $e_k$ is the unit vector having the only non-zero component on the $k$-th position.

\begin{lem}\label{second good operator fact} Let $(T_k)_{1\leq k\leq K}\subset\mathcal{L}_{1,\infty}$ be such that 
$$\lim_{t\to\infty}t\mu(t,T_k)=\alpha_k,\quad 1\leq k\leq K.$$
It follows that
$$\lim_{t\to\infty}t\mu\Big(t,\bigoplus_{1\leq k\leq K}T_k\Big)=\sum_{1\leq k\leq K}\alpha_k.$$
\end{lem}
\begin{proof} Let $z(n)=\frac1{n+1}.$ For every $1\leq k\leq K,$ choose $S_k\in\mathcal{L}_{1,\infty}$ such that
$$\mu(S_k)=\alpha_k z,$$
and such that $S_k-T_k\in(\mathcal{L}_{1,\infty})_0.$ We have
$$\mu\Big(\bigoplus_{1\leq k\leq K}S_k\Big)=\mu\Big(z\otimes \big\{\alpha_k\big\}_{1\leq k\leq K}\Big).$$
It follows from Lemma \ref{tensor lemma} that
$$\lim_{t\to\infty}t\mu\Big(t,\bigoplus_{1\leq k\leq K}S_k\Big)=\sum_{1\leq k\leq K}\alpha_k.$$
On the other hand, we have
$$\bigoplus_{1\leq k\leq K}S_k-\bigoplus_{1\leq k\leq K}T_k\in(\mathcal{L}_{1,\infty})_0.$$
The assertion follows now from Lemma \ref{first good operator fact}.	
\end{proof}

\section{Proof of Theorem \ref{cif manifold}}\label{cif manifold section}

In this section, we prove an asymptotic formula for singular values in Theorem \ref{cif manifold}. It is modeled after the proof of Lemma 1 in \cite{Birman-Solomyak-SMZh}.

We refer the reader to \cite{RuzhanskyTurunen} for the theory of pseudo-differential operators.

\begin{defi} Pseudo-differential operator $Q:\mathcal{S}(\mathbb{R}^d)\to\mathcal{S}(\mathbb{R}^d)$ is called compactly supported if there exists $\phi\in C^{\infty}_c(\mathbb{R}^d)$ such that $M_{\phi}QM_{\phi}=Q.$
\end{defi}

The definition below should be compared with the Definition 10.2.24 in \cite{LSZ-book}.
\begin{defi}\label{classical psdo def} Pseudo-differential operator $Q:\mathcal{S}(\mathbb{R}^d)\to\mathcal{S}(\mathbb{R}^d)$ with the symbol $q$ of order ${\rm ord}(Q)$ is called classical if there exists a sequence $(q_n)_{n\leq{\rm ord}(Q)}$ of functions on $\mathbb{R}^d\times\mathbb{R}^d$ and sequence $(m_n)_{n\leq {\rm ord}(Q)}$ of real numbers such that
\begin{enumerate}[{\rm (i)}]
\item sequence $(m_n)_{n\leq{\rm ord}(Q)}$ is strictly increasing, $m_{{\rm ord}(Q)}={\rm ord}(Q)$ and $m_n\to-\infty$ as $n\to-\infty;$
\item for every $n\leq{\rm ord}(Q),$ $q_n$ is homogeneous of degree $m_n$ in the second variable;
\item for every $n\leq{\rm ord}(Q),$ $q_n\in C^{\infty}(\mathbb{R}^d\times\mathbb{S}^{d-1})$ (i.e., function and all its derivatives are bounded); 
\item for every $k\leq{\rm ord}(Q),$ we have
$$q-\sum_{n=k}^{{\rm ord}(Q)}q_n\cdot (1-\phi)$$
is a symbol of order $m_{k-1};$
\end{enumerate}	
Here, $\phi:\mathbb{R}^d\times\mathbb{R}^d\to\mathbb{C}$ does not depend on the first variable, is compactly supported in the second variable and equals $1$ near $0.$ 	
\end{defi}

The key ingredient is a re-statement of Theorems 1 and 2 in \cite{Birman-Solomyak-vestnik-1977}. In the notation of \cite{Birman-Solomyak-vestnik-1977}, $m=d,$ $\mu=1;$ $a$ and $c$ are smooth functions on $\mathbb{R}^d$ compactly supported in the cube; $b$ smooth (except at $0$) function on $\mathbb{R}^d\times\mathbb{R}^d$ compactly supported in the first argument and homogeneous of degree $-d$ in the second argument (so that $\gamma=1$ and $\tau_1=\cdots=\tau_d=1$). The fact that $a_{\gamma}$ is a Schur multiplier on $\mathcal{L}_{\beta,\infty}$ follows by writing $a|_{\mathbb{R}^d\times\mathbb{S}^{d-1}}$ as Fourier series in spherical functions (see a similar argument in Lemma 8.1 in \cite{SZ-DAO1}). That is, we are in the conditions of the third part of Theorem 1 in \cite{Birman-Solomyak-vestnik-1977}. That theorem together with Theorem 2 in \cite{Birman-Solomyak-vestnik-1977} yields the assertion below.

\begin{thm}\label{birsol psdo theorem} Let $Q$ be a classical compactly supported pseudo-differential operator of order $-d$ on $\mathbb{R}^d.$ If $Q$ is self-adjoint, then there exists a limit
$$\lim_{t\to\infty}t\mu(t,Q).$$
\end{thm}
\begin{proof} By Definition \ref{classical psdo def}, we can find a smooth (except at $0$) homogeneous of degree $-d$ (in the second variable) function $q_{-d}:\mathbb{R}^d\times\mathbb{R}^d\to\mathbb{C},$ a smooth compactly supported function $\phi$ on $\mathbb{R}^d$ such that $\phi=1$ near $0$ and a pseudo-differential operator $R$ of order $-d-\epsilon$ such that
$$Q=T_{q_{-d}}\cdot (1-\phi)(\nabla)+R.$$
Since $Q$ is compactly supported, it follows that there exists smooth compactly supported function $\psi$ on $\mathbb{R}^d$ such that
$$Q=M_{\psi}QM_{\psi}=M_{\psi}T_{q_{-d}}(1-\phi)(\nabla)M_{\psi}+M_{\psi}RM_{\psi}.$$
Since $Q$ is self-adjoint, it follows that
$$Q=\Re(M_{\psi}T_{q_{-d}}(1-\phi)(\nabla)M_{\psi})+\Re(M_{\psi}RM_{\psi}).$$
Note that $P=R(1-\Delta)^{\frac{d+\epsilon}{2}}$ is a pseudo-differential operator of order $0.$ Hence, $M_{\psi}P$ is bounded.
$$M_{\psi}RM_{\psi}=M_{\psi}P\cdot (1-\Delta)^{-\frac{d+\epsilon}{2}}M_{\psi}\in\mathcal{L}_{\frac{d}{d+\epsilon},\infty}\subset(\mathcal{L}_{1,\infty})_0.$$
Theorem 2 in \cite{Birman-Solomyak-vestnik-1977} asserts the existence of the limit
$$\lim_{t\to\infty}t\mu\big(t,\Re(M_{\psi}T_{q_{-d}}(1-\phi)(\nabla)M_{\psi})\big).$$
The assertion follows now from Lemma \ref{first good operator fact}.	
\end{proof}

Now, we want a similar result for compact manifolds.

\begin{defi} Pseudo-differential operator $P:\mathcal{S}(X,g)\to\mathcal{S}(X,g)$ is called compactly supported in the chart $(U,\gamma)$ if there exists a smooth function $\phi$ compactly supported in this chart such that  $M_{\phi}PM_{\phi}=P.$
\end{defi}

\begin{defi} Pseudo-differential operator $P:\mathcal{S}(X,g)\to\mathcal{S}(X,g)$ is called classical if, for every chart $(U,\gamma)$ and for every smooth function $\phi$ compactly supported in this chart, the operator $M_{\phi}PM_{\phi}$ becomes classical when expressed in local coordinates.
\end{defi}

Let $(U,\gamma)$ be a chart. A linear mapping $W_{\gamma}:L_2(\mathbb{R}^d)\to L_2(X,{\rm vol}_g)$ defined by the formula
$$W_{\gamma}\xi=\chi_U\cdot ((\xi\cdot {\rm det}(g^{-\frac12}))\circ\gamma),\quad \xi\in L_2(\mathbb{R}^d),$$
is an isometry.

\begin{lem}\label{manifold asymptotic csup lemma} Let $(X,g)$ be a compact $d$-dimensional Riemannian manifold. Let $P$ be a classical pseudo-differential operator of order $-d$ on $(X,g).$ If $P$ is self-adjoint and compactly supported in the chart $(U,\gamma),$ then there exists a limit
$$\lim_{t\to\infty}t\mu(t,P).$$
\end{lem}
\begin{proof} It is immediate that $\mu(P)=\mu(Q),$ where $Q=W_{\gamma}^{\ast}PW_{\gamma}.$ Since $P$ is a pseudo-differential operator of order $-d$ compactly supported in the chart $(U,\gamma),$ it follows that $Q$ is a compactly supported pseudo-differential operator of order $-d$ on $\mathbb{R}^d.$ $P$ is classical, hence so is $Q.$ The assertion follows now from Theorem \ref{birsol psdo theorem}.
\end{proof}

\begin{lem}\label{manifold asymptotic main lemma} Let $(X,g)$ be a compact $d$-dimensional Riemannian manifold. Let $P$ be a classical pseudo-differential operator of order $-d$ on $(X,g).$ If $P$ is self-adjoint, then there exists a limit
$$\lim_{t\to\infty}t\mu(t,P).$$
\end{lem}
\begin{proof} Fix a finite atlas on $(X,g).$ Choose a sequence $(\phi_{n,k})_{\substack{1\leq k\leq K\\ n\in\mathbb{Z}}}\subset C^{\infty}(X)$ such that
\begin{enumerate} 
\item we have
$$\phi_n\stackrel{def}{=}\sum_{1\leq k\leq K}\phi_{n,K}\to 1$$
in $L_2(X,{\rm vol}_g).$
\item for every $n\in\mathbb{Z}$ and for every $1\leq k\leq K,$  $\phi_{n,k}$ is compactly supported in some chart.
\item for every $n\in\mathbb{Z},$ the sets $({\rm supp}(\phi_{n,k}))_{1\leq k\leq K}$ are separated.
\end{enumerate}
	
For every $n\in\mathbb{Z},$ we write
$$M_{\phi_n}PM_{\phi_n}=\sum_{1\leq k\leq K}M_{\phi_{n,k}}PM_{\phi_{n,k}}+\sum_{\substack{1\leq k_1,k_2\leq K\\ k_1\neq k_2}}M_{\phi_{n,k_1}}PM_{\phi_{n,k_2}}\stackrel{def}{=}A_n+B_n.$$
	
If $k_1\neq k_2,$ then
$$M_{\phi_{n,k_1}}PM_{\phi_{n,k_2}}=[M_{\phi_{n,k_1}},P]M_{\phi_{n,k_2}}$$
is a pseudo-differential operator of order $-d-1.$ In particular, it belongs to $\mathcal{L}_{\frac{d}{d+1},\infty}.$ Thus, $B_n\in(\mathcal{L}_{1,\infty})_0.$

On the other hand, we have
$$\mu(A_n)=\mu\Big(\bigoplus_{1\leq k\leq K}M_{\phi_{n,k}}PM_{\phi_{n,k}}\Big).$$
Each operator $M_{\phi_{n,k}}PM_{\phi_{n,k}}$ is self-adjoint, classical and compactly supported in some chart. By Lemma \ref{manifold asymptotic csup lemma}, there exists a limit
$$\lim_{t\to\infty}t\mu(t,M_{\phi_{n,k}}PM_{\phi_{n,k}}).$$
By Fact \ref{second good operator fact}, there exists a limit
\begin{equation}\label{maml eq2}
\lim_{t\to\infty}t\mu(t,A_n).
\end{equation}
Recall that  $B_n\in(\mathcal{L}_{1,\infty})_0.$ Combining \eqref{maml eq2} and Corollary \ref{first good operator fact}, we infer the existence of the limit
\begin{equation}\label{maml eq3}
\lim_{t\to\infty}t\mu(t,M_{\phi_n}PM_{\phi_n}).
\end{equation}
	
Now,
$$\|P-M_{\phi_n}PM_{\phi_n}\|_{1,\infty}\leq 2\|M_{1-\phi_n}P\|_{1,\infty}+2\|PM_{1-\phi_n}\|_{1,\infty}\leq $$
$$\leq 2\Big\|M_{1-\phi_n}(1-\Delta_g)^{-\frac{d}{2}}\Big\|_{1,\infty}\cdot\Big(\Big\|(1-\Delta_g)^{\frac{d}{2}}P\Big\|_{\infty}+\Big\|P(1-\Delta_g)^{\frac{d}{2}}\Big\|_{\infty}\Big)\leq$$
$$\leq c_{X,g}\|1-\phi_n\|_2.$$
Thus, $M_{\phi_n}PM_{\phi_n}\to P$ in $\mathcal{L}_{1,\infty}.$ The assertion follows now by combining \eqref{maml eq3} and Corollary \ref{limit of good operators corollary}.
\end{proof}

The proof of the next lemma is based on a deep result from \cite{DFWW}.

\begin{lem}\label{bounded cif manifold} Let $(X,g)$ be a compact $d$-dimensional Riemannian manifold and let $\Delta_g$ be the Laplace-Beltrami operator. For every $f\in L_{\infty}(X,{\rm vol}_g)$ and for every normalised continuous trace on $\mathcal{L}_{1,\infty},$ we have
$$\varphi((1-\Delta_g)^{-\frac{d}{4}}M_f(1-\Delta_g)^{-\frac{d}{4}})=\frac{{\rm Vol}(\mathbb{S}^{d-1})}{d(2\pi)^d}\int_X fd{\rm vol}_g.$$
\end{lem}
\begin{proof} Let 
$$A=(1-\Delta_g)^{-\frac{d}{4}},\quad B=M_f(1-\Delta_g)^{-\frac{d}{4}}.$$
By Theorem \ref{solomyak cwikel estimate manifold}, we have $A,B\in\mathcal{L}_{2,\infty}.$ In particular, $AB,BA\in\mathcal{L}_{1,\infty}$ and
$$[A,B]\stackrel{def}{=}AB-BA\in [\mathcal{L}_{2,\infty},\mathcal{L}_{2,\infty}].$$
Here, the notation $[\mathcal{I},\mathcal{J}]$ stands for the linear span of all commutators $[X,Y],$ $X\in\mathcal{I},$ $Y\in\mathcal{J}.$ It is a deep result, proved in \cite{DFWW} (see e.g. p.3 there) that
$$[\mathcal{L}_{2,\infty},\mathcal{L}_{2,\infty}]=[\mathcal{L}_{1,\infty},\mathcal{L}_{\infty}].$$
Since $\varphi$ vanishes on $[\mathcal{L}_{1,\infty},\mathcal{L}_{\infty}],$ it follows that $\varphi$ vanishes on $[\mathcal{L}_{2,\infty},\mathcal{L}_{2,\infty}].$ In particular, $\varphi([A,B])=0.$
	
Hence,
$$\varphi((1-\Delta_g)^{-\frac{d}{4}}M_f(1-\Delta_g)^{-\frac{d}{4}})=\varphi(AB)=$$
$$=\varphi(BA)=\varphi(M_f(1-\Delta_g)^{-\frac{d}{2}}).$$
The assertion follows now from Theorem 11.7.10 in \cite{LSZ-book}.
\end{proof}

Based on Lemma \ref{manifold asymptotic main lemma} and Lemma \ref{bounded cif manifold}, we are able to prove Theorem \ref{cif manifold} for positive bounded functions.

\begin{lem}\label{cif bounded positive lemma} Let $(X,g)$ be a compact $d$-dimensional Riemannian manifold. Let $0\leq f\in L_{\infty}(X,{\rm vol}_g).$ We have
$$\lim_{t\to\infty}t\mu\Big(t,(1-\Delta_g)^{-\frac{d}{4}}M_f(1-\Delta_g)^{-\frac{d}{4}}\Big)=\frac{{\rm Vol}(\mathbb{S}^{d-1})}{d(2\pi)^d}\int_X fd{\rm vol}_g.$$
\end{lem}
\begin{proof} Choose a sequence $(f_n)_{n\geq0}\subset C^{\infty}(X,\mathbb{R})$ such that $f_n\to f$ in $L_M(X,{\rm vol}_g).$ By Theorem \ref{solomyak cwikel estimate manifold}, we have
$$(1-\Delta_g)^{-\frac{d}{4}}M_{f_n}(1-\Delta_g)^{-\frac{d}{4}}\to (1-\Delta_g)^{-\frac{d}{4}}M_f(1-\Delta_g)^{-\frac{d}{4}}$$
in $\mathcal{L}_{1,\infty}$ as $n\to\infty.$ Operators
$$(1-\Delta_g)^{-\frac{d}{4}}M_{f_n}(1-\Delta_g)^{-\frac{d}{4}},\quad n\geq0,$$
are self-adjoint and pseudo-differential. By Lemma \ref{manifold asymptotic main lemma}, there exists a limit
$$\lim_{t\to\infty}t\mu\Big(t,(1-\Delta_g)^{-\frac{d}{4}}M_{f_n}(1-\Delta_g)^{-\frac{d}{4}}\Big)=\alpha_n.$$
By Lemma \ref{limit of good operators mu lemma}, there exists a limit
$$\lim_{t\to\infty}t\mu\Big(t,(1-\Delta_g)^{-\frac{d}{4}}M_f(1-\Delta_g)^{-\frac{d}{4}}\Big)=\alpha.$$

Hence, for every continuous normalised trace $\varphi$ on $\mathcal{L}_{1,\infty},$ we have
$$\varphi\Big((1-\Delta_g)^{-\frac{d}{4}}M_f(1-\Delta_g)^{-\frac{d}{4}}\Big)=\alpha.$$
The assertion follows now from Lemma \ref{bounded cif manifold}.
\end{proof}

The proof of the next lemma exploits recent advances in Birman-Koplienko-Solomyak inequality for quasi-Banach ideals.

\begin{lem}\label{positive part abstract lemma} Let $A\in\mathcal{L}_{2,\infty}$ and $B\in\mathcal{L}_{\infty}$ be self-adjoint elements. If $[A,B]\in(\mathcal{L}_{2,\infty})_0,$ then
$$(ABA)_+-AB_+A\in(\mathcal{L}_{1,\infty})_0.$$
\end{lem}
\begin{proof} We have 
\begin{align*}
|ABA|^2-(A|B|A)^2&=(ABA)^2-(A|B|A)^2\\&
=A\cdot BA\cdot AB\cdot A- A\cdot |B|A\cdot A|B|\cdot A\\&
=A\cdot[B,A]\cdot AB\cdot A+A\cdot AB\cdot [A,B]\cdot A\\&
-A\cdot [|B|,A]\cdot A|B|\cdot A-A\cdot A|B|\cdot [A,|B|]\cdot A.
\end{align*}

It follows from the assumption $[A,B]\in(\mathcal{L}_{2,\infty})_0$ and from Theorem 3.4 in \cite{DDPS1} that $[A,|B|]\in(\mathcal{L}_{2,\infty})_0.$ By H\"older inequality, we have
$$|ABA|^2-(A|B|A)^2\in(\mathcal{L}_{\frac12,\infty})_0.$$
Now recall the Birman-Koplienko-Solomyak inequality for quasi-Banach ideals (see Theorem 6.3 in \cite{HSZ}). Applying the latter theorem with $E=(\mathcal{L}_{2,\infty})_0,$ $f(t)=|t|^{\theta},$ $t\in\mathbb{R},$ and $p=\theta=\frac12,$ we obtain
$$|ABA|-A|B|A\in(\mathcal{L}_{1,\infty})_0.$$
Thus,
$$(ABA)_+-AB_+A=\frac{ABA+|ABA|}{2}-\frac{ABA+A|B|A}{2}=$$
$$=\frac12(|ABA|-A|B|A)\in (\mathcal{L}_{1,\infty})_0.$$
\end{proof}

\begin{lem}\label{commutator small lemma} Let $(X,g)$ be a compact $d$-dimensional Riemannian manifold. Let $f\in L_{\infty}(X,{\rm vol}_g).$ We have
$$[M_f,(1-\Delta_g)^{-\frac{d}{4}}]\in(\mathcal{L}_{2,\infty})_0.$$
\end{lem}
\begin{proof} Choose a sequence $(f_n)_{n\geq0}\subset C^{\infty}(X,\mathbb{R})$ such that $f_n\to f$ in $L_M(X,{\rm vol}_g).$ By Theorem \ref{solomyak cwikel estimate manifold}, we have
$$[M_{f_n},(1-\Delta_g)^{-\frac{d}{4}}]\to [M_f,(1-\Delta_g)^{-\frac{d}{4}}]$$
in $\mathcal{L}_{2,\infty}$ as $n\to\infty.$ The operators
$$[M_{f_n},(1-\Delta_g)^{-\frac{d}{4}}],\quad n\geq0,$$
are pseudo-differential of order $-\frac{d}{2}-1.$ In particular, they belong to $(\mathcal{L}_{2,\infty})_0.$ The assertion follows immediately.
\end{proof}

\begin{proof}[Proof of Theorem \ref{cif manifold}] Suppose first $f$ is bounded. By Lemma \ref{commutator small lemma} and Lemma \ref{positive part abstract lemma}, we have
$$\Big((1-\Delta_g)^{-\frac{d}{4}}M_f(1-\Delta_g)^{-\frac{d}{4}}\Big)_+-(1-\Delta_g)^{-\frac{d}{4}}M_{f_+}(1-\Delta_g)^{-\frac{d}{4}}\in(\mathcal{L}_{1,\infty})_0.$$
By Lemma \ref{cif bounded positive lemma} and Lemma \ref{first good operator fact}, we have
$$\lim_{t\to\infty}t\mu\Big(t,\Big((1-\Delta_g)^{-\frac{d}{4}}M_f(1-\Delta_g)^{-\frac{d}{4}}\Big)_+\Big)=\frac{{\rm Vol}(\mathbb{S}^{d-1})}{d(2\pi)^d}\int_X f_+d{\rm vol}_g.$$ 
Applying the latter formula to $-f,$ we obtain
$$\lim_{t\to\infty}t\mu\Big(t,\Big((1-\Delta_g)^{-\frac{d}{4}}M_f(1-\Delta_g)^{-\frac{d}{4}}\Big)_-\Big)=\frac{{\rm Vol}(\mathbb{S}^{d-1})}{d(2\pi)^d}\int_X f_-d{\rm vol}_g.$$ 
This proves the assertion for bounded $f.$

Suppose now $f\in L_M(X,{\rm vol}_g).$ Choose a sequence $(f_n)_{n\geq0}\subset L_{\infty}(X,{\rm vol}_g)$ such that $f_n\to f$ in $L_M(X,{\rm vol}_g).$ The assertion follows from the preceding paragraph (applied to $f_n$) and Corollary \ref{limit of good operators corollary}.

\end{proof}

\end{document}